\documentclass[11pt]{amsart}
\usepackage{geometry}                
\usepackage{graphicx}
\usepackage{amssymb}
\usepackage{epstopdf}
\usepackage{float}
\usepackage{amsmath,amsfonts,amsthm,url,array,etoolbox}
\usepackage{enumerate}
\usepackage{tikz}
\usetikzlibrary{arrows,positioning,calc,decorations.markings}
\theoremstyle{plain}
\newtheorem{thm}{Theorem}
\newtheorem{lem}[thm]{Lemma}

\theoremstyle{definition}
\newtheorem{defn}{Definition}

\newtheorem*{exmp*}{Example}

\theoremstyle{remark}

\tikzset{
    >=stealth',
    pil/.style={
           ->,
           thick}
}

\begin{document}
\title[]{Delone property of the holonomy vectors of translation surfaces}
\author[]{Chenxi Wu}
\maketitle
\begin{abstract}
We answer a question by Barak Weiss on the uniform discreteness of the set of the holonomy vectors of translation surfaces.
\end{abstract}

For a translation surface $M$, let $S_M\subset\mathbb{R}^2$ be the set of holonomy vectors of all saddle connections of $M$. The following is a plot of $S_M$ where $M$ is a lattice surface in $H(2)$ with discriminant 13.

\begin{figure}[H]
\includegraphics[scale=0.3]{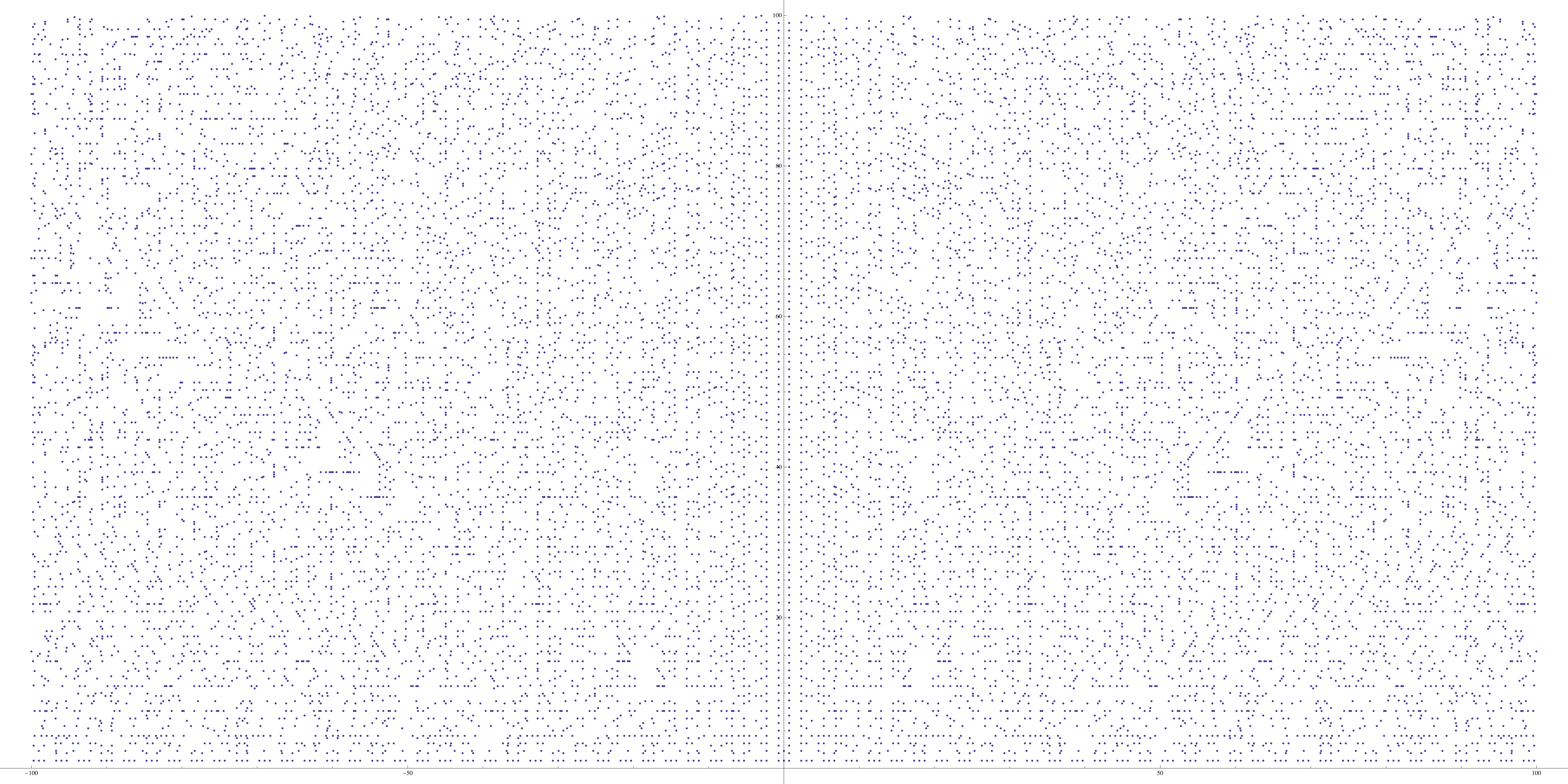}
\end{figure} 

A way of capturing the concept of uniformity of a subset of a metric space is the concept of Delon\'e set. 
\begin{defn}\cite{senechal2006quasicrystal}
Subset $A$ of a metric space $X$ is a Delon\'e set if:\\
\begin{enumerate}
\item $A$ is relatively dense, i.e. there is $R>0$ such that any ball of radius $R$ in $X$ contains at least one point in $A$.
\item $A$ is uniformly discrete, i.e. there is $r>0$ such that for any two distinct points $x,y\in A$, $d(x,y)>r$.
\end{enumerate}
\end{defn}

The Delon\'e property implies a quadratic upper and lower bound on the growth rate of $S_M$ (the number of points within radius $R$ of the origin). It is known that the growth rate of $S_M$ does satisfy such upper and lower bounds by \cite{masur1988lower} and \cite{masur1990growth}. For some translation surfaces the asymptotic upper and lower bounds agree. Veech \cite{veech1989teichmuller}, Eskin and Masur \cite{eskin2001asymptotic} showed that this is the case for Veech surfaces and generic translation surfaces respectively. Also, whether or not $M$ is a lattice surface is determined by properties of $S_M$ as shown in Smillie and Weiss \cite{smillie2010characterizations}. Furthermore, additional properties that $S_M$ has to satisfy are contained in the works of  Athreya and Chaika \cite{athreya2012distribution,athreya2013gap} on the distribution of angles between successive saddle connections of bounded length.\\

Barak Weiss asks for which translation surfaces $M$, is $S_M$ a Delon\'e set. Here, we will show that:

\begin{thm} If $M$ is a lattice surface then $S_M$ is never a Delon\'e set. On the other hand, there exists a non-lattice translation surface $M$ for which $S_M$ is a Delon\'e set.
\end{thm}

We will show that if $M$ is a non-arithmetic lattice surface, then $S_M$ cannot be uniformly discrete. We will also show that if $M$ is square-tiled then $S_M$ cannot be relatively dense. Combining these two results, we can conclude that when $M$ is a lattice surface, $S_M$ cannot be a Delon\'e set. \\

\begin{thm} If $M$ is a non-arithmetic lattice surface, then $S_M$ is not uniformly discrete.
\end{thm}

\begin{proof}
Let $M$ is a non-arithmetic lattice surface and let $r>0$ be given.
By \cite{veech1989teichmuller} we can choose a periodic direction $\gamma$ of $M$ such that in this direction $M$ is decomposed into cylinders $C_1,\dots,C_n$ in the direction of $\gamma$, and the width of all these cylinders are no larger than $r/4$. Because $M$ is a lattice surface, the holonomy field \cite{ks} is generated by the ratios of the circumferences. Because $M$ is not square-tiled, the holonomy field can not be $\mathbb{Q}$. Hence, there exist two numbers $i$ and $j$ such that the quotient of the circumferences of $C_i$ and $C_j$ is not in $\mathbb{Q}$. Denote the holonomy vectors of periodic geodesics corresponding to $C_i$ and $C_j$ by $l$ and $l'$. Let $h$ and $h'$ be the holonomy vectors of two saddle connections $\alpha_0$ and $\beta_0$ crossing $C_i$ and $C_j$ respectively. Let $\alpha_n$ be the images of $\alpha_0$ under $n$-Dehn twists in cylinders $C_i$, $\beta_n$ be the image of $\beta_0$ under $n$-Dehn twists in cylinder $C_j$. Given $n\in\mathbb{Z}$, both $\alpha_n$ and $\beta_n$ are still saddle connections of $M$. Thus, for any integer $n$, the vectors $h+nl$ and $h'+nl'$ are in $S_M$. \\

\begin{figure}[H]
\begin{center}
\begin{tikzpicture}
\draw[-](0,0)--(0,2)--(6,2)--(6,0)--(0,0)node[pos=.5,below]{the cylinder $C_i$};
\draw[pil](0,2.5)--(6,2.5)node[pos=.5,above]{$l$};
\draw[pil](5,2)--(2,0)node[pos=.5,left]{$\alpha_0$};
  \fill (5,2) circle (1.5pt);
  \fill (2,0) circle (1.5pt);
\draw[pil](5,2)--(6,1.33)node[pos=.5,right]{$\alpha_1$};
\draw[pil](0,1.33)--(2,0)node[pos=.5,right]{$\alpha_1$};
\draw[pil](3,-1)--(0,-3)node[pos=.5,left]{$h$};
\draw[pil](3,-1)--(6,-3)node[pos=.5,left]{$h+l$};
\draw[pil](3,-1)--(12,-3)node[pos=.5,left]{$h+2l$};
\node at (6,-3.5){$\text{the holonomy vector of $\alpha_0$ and its images under Dehn twists}$};
\end{tikzpicture}
\end{center}
\end{figure}
\begin{figure}[H]
\begin{center}
\begin{tikzpicture}
\draw[-](0,-5)--(0,-4)--(5.245,-4)--(5.245,-5)--(0,-5)node[pos=.5,below]{the cylinder $C_j$};
\draw[pil](0,-3.5)--(5.245,-3.5)node[pos=.5,above]{$l'$};
\draw[pil](4,-4)--(3,-5)node[pos=.5,left]{$\beta_0$};

  \fill (4,-4) circle (1.5pt);
  \fill (3,-5) circle (1.5pt);
\draw[pil](4,-4)--(5.245,-4-1.245/4.245)node[pos=.5,right]{$\beta_1$};
\draw[pil](0,-4-1.245/4.245)--(3,-5)node[pos=.5,right]{$\beta_1$};
\draw[pil](1,-6)--(0,-7)node[pos=.5,left]{$h'$};
\draw[pil](1,-6)--(5.245,-7)node[pos=.5,left]{$h'+l'$};
\draw[pil](1,-6)--(10.49,-7)node[pos=.5,left]{$h'+2l'$};
\node at (5,-7.5){$\text{the holonomy vectors of $\beta_0$ and its images under Dehn twists}$};
\end{tikzpicture}
\end{center}
\end{figure}

Let us write $h$ as $h=h_1+h_2$ and $h'$ as $h'=h'_1+h'_2$, where $h_1$, $h'_1$ are vectors in direction $\gamma$, and $h_2$, $h'_2$ are vectors in direction $\gamma^{\perp}$. Because the width of $C_i$ and $C_j$ are no larger than $r/4$ by assumption, $||h_2-h'_2||<r/2$. Because $h_1$, $h'_1$, $l$ and $l'$ are vectors pointing in the same direction, we can write $h_1=al$, $h'_1=bl$, $l'=\lambda l$. Because the quotient $\lambda$ of lengths of $l'$ and $l$ is irrational, the set $\{m+m'\lambda:m,m'\in\mathbb{Z}\}$ is dense in $\mathbb{R}$, so there exists a pair of integers $n_0$ and $n'_0$ such that $|a-b+n_0-n'_0\lambda|<{r\over 2||l||}$. Thus $||(h_1+n_0l)-(h'_1+n'_0l')||<r/2$, and $||(h+n_0l)-(h'+n'_0l')||<r/2+r/2=r$. We conclude that $S_M$ contains two points for which the distance between them is less than any $r>0$, thus $S_M$ is not uniformly discrete.
\end{proof}

The above argument also works on those completely periodic surfaces such that in any given periodic direction, there are at least two closed geodesics whose length are not related by a rational multiple. Barak Weiss pointed out that the above argument implies that any orbit in $\mathbb{R}^2$ under a lattice in $SL(2,\mathbb{R})$ is either contained in a lattice in $\mathbb{R}^2$ or not uniformly discrete.\\

Now we deal with the square-tiled case. This case is closely related to the classical case of the torus discussed in \cite{herzog1971patterns}. In this case, $S_M=\{(p,q):\gcd(p,q)=1\}$\\

\begin{lem} For any positive integer $N$, the set $\{(p,q)\in\mathbb{Z}^2:\gcd(p,q)\leq N\}$ is not relatively dense in $\mathbb{R}^2$.\end{lem}

\begin{proof}(\cite{herzog1971patterns}) Given $R>0$, choose an integer $n>2R$, and $n^2$ distinct prime numbers $p_{i,j}$, $1<i,j<n$ larger than $N$. Let $q_i=\prod_j p_{i,j}$ and $q'_j=\prod_i p_{i,j}$. By the Chinese remainder theorem there is an integer $x$ such that for all $i$, $x\equiv -i \mod{q_i}$, and an integer $y$ such that for all $j$, $y\equiv -j \mod{q'_j}$. Hence for any two positive integers $i,j\leq n$, $(x+i,y+j)\in \{(p,q)\in\mathbb{Z}^2:\gcd(p,q)>N\}$, in particular, there is a ball in $\mathbb{R}^2$ of radius $R$ that does not contain points in the set $\{(p,q)\in\mathbb{Z}^2:\gcd(p,q)\leq N\}$.
\end{proof}

Now we use the above lemma to show that if the surface $M$ is square-tiled, $S_M$ cannot be relatively dense.\\

\begin{thm} If $M$ is a square-tiled lattice surface then $S_M$ is not relatively dense in $\mathbb{R}^2$.\end{thm}

\begin{proof}
If $M$ is square-tiled, we can assume that $M$ is tiled be $1\times 1$ squares. Let $N$ be the number of squares that tiled $M$, then $M$ is an $n$-fold branched cover of $T=\mathbb{R}^2/\mathbb{Z}^2$ branched at $(0,0)$. Therefore, the holonomy of any saddle connection is in $\mathbb{Z}\times\mathbb{Z}$. For any pair of coprime integers $(p,q)$, let $\gamma$ be the closed geodesic in $T$ starting at $(0,0)$ whose holonomy is $(p,q)$. The length of $\gamma$ is $\sqrt{p^2+q^2}$. The preimage of $\gamma$ in $M$ is a graph $\Gamma$. The vertices of $\Gamma$ are the preimages of $(0,0)$, while edges are the preimages of $\gamma$. The sum of the lengths of the edges  of $\Gamma$ is $N\sqrt{p^2+q^2}$. Any saddle connection of $M$ in $(p,q)$-direction is a path on $\Gamma$ without self intersection, hence The length of such a saddle connection can not be greater than $N\sqrt{p^2+q^2}$. Hence, the holonomy of such a saddle connection is of the form $(sp,sq)$, $s\in\mathbb{Z}$ with $|s|\leq N$. Thus, $S_M\subset\{(p,q)\in\mathbb{Z}^2:\gcd(p,q)\leq N\}$, so by Lemma 3 $S_M$ is not relatively dense in $\mathbb{R}^2$.
\end{proof}

Finally, when $M$ is not a lattice surface, $S_M$ \emph{can} be a Delon\'e set, which we will show in Example 1 below. This construction finishes the proof of Theorem 1.\\

\begin{exmp*}  Let $M_1$ be the branched double cover of $\mathbb{R}^2/\mathbb{Z}^2$ branched at points $(0,0)$ and $(\sqrt{2}-1,\sqrt{3}-1)$. 
Let $\tilde{M}$ be a $\mathbb{Z}^2$-cover which is a branched double cover of $\mathbb{R}^2$ branched at $U=\mathbb{Z}^2$ and at $V=\mathbb{Z}^2+(\sqrt{2}-1,\sqrt{3}-1)$, where the deck group action is by translation. Then saddle connections on $M_1$ lift to saddle connections on $\tilde{M}$, and any two lifts have the same holonomy, hence $S_{M_1}$ is the same as $S_{\tilde{M}}$, which is the set of holonomies of line segments linking two points in $W=U\cup V$ which do not pass through any other point in $W$. If a line segment links two points in $U$, its slope must be rational or $\infty$, hence it would not pass through any point in $V$. Furthermore, it does not pass through any other point in $U$ if and only if its holonomy is a pair of coprime integers. The same is true for line segments linking two points in $V$. On the other hand, given any point $p\in U$ and any point $q\in V$, a line segment from $p$ to $q$ has irrational slope hence cannot pass through any other point in $U$ or $V$, therefore the holonomy of such line segment can be any vector in $\mathbb{Z}^2+(\sqrt{2}-1,\sqrt{3}-1)$. Similarly the holonomies of  saddle connections from $V$ to $U$ are $\mathbb{Z}^2-(\sqrt{2}-1,\sqrt{3}-1)$. Hence $S_{M_1}=\{(a,b)\in\mathbb{Z}^2:\gcd(a,b)=1\}\cup(\mathbb{Z}^2+(\sqrt{2}-1,\sqrt{3}-1))\cup(\mathbb{Z}^2-(\sqrt{2}-1,\sqrt{3}-1))$, this set is uniformly discrete, and the last two pieces are uniformly dense.\end{exmp*}

\textbf{Questions:}
\begin{enumerate}
\item Can the set of holonomy vectors of all saddle connections of a non-arithmetic lattice surface be relatively dense? 
\item Is there any characterization of the set of flat surfaces $M$ such that $S_M$ are Delon\'e, relatively dense or uniformly discrete? 
\item Is there a surface $M$ which is not a branched cover of the torus for which $S_M$ is Delon\'e?
\end{enumerate}

\bibliographystyle{alpha} 
\bibliography{nud}
\end{document}